\documentclass[11pt]{amsart}
\usepackage{amssymb}
\usepackage[all]{xy}
\newtheorem{theorem}{Theorem}[section]
\newtheorem{lemma}[theorem]{Lemma}
\newtheorem{proposition}[theorem]{Proposition}
\newtheorem{corollary}[theorem]{Corollary}

\theoremstyle{definition}
\newtheorem{definition}[theorem]{Definition}

\theoremstyle{remark}
\newtheorem{remark}[theorem]{Remark}

\numberwithin{equation}{section}

\begin{document}

%%%%%%%%%%%%%%%%%%%%  MATH  LETTERS   %%%%%%%%%%%%%%%%%%%%

\def\Ga{\Gamma}   
\def\Om{\Omega}   
\def\al{\alpha}     
\def\be{\beta}       
\def\ga{\gamma}      
\def\la{\lambda}     

\def\mM{\mathbb{M}} 

\def\dM{\mathfrak{M}} 
\def\dT{\mathfrak{T}}

 \def\kS{\mathcal{S}}
\def\kI{\mathcal{I}} 
\def\kM{\mathcal{M}} 

\def\tiv{\tilde{V}}	  
\def\tph{\tilde{\phi}}
\def\tal{\tilde{\al}}   
\def\tbe{\tilde{\be}}  
\def\tiw{\tilde{W}}	
\def\tps{\tilde{\psi}}

%%%%%%%%%%%%%%%%%%%%%%%%%%%%%%%%%%%%%%%%%%%%%%%%%%%%%%

\def\xx{\times}
\def\*{\otimes}		
\def\+{\oplus}		
 \def\sbe{\subseteq}
\def\0{\emptyset}
\def\cc{\boldsymbol\cdot}
\def\ito{\stackrel\sim\to}
\def\tto{\twoheadrightarrow}
\def\cir{\mathop{\scriptstyle\circ}}

\def\mtr#1{\begin{pmatrix}#1\end{pmatrix}}
\def\smtr#1{\left(\begin{smallmatrix}#1\end{smallmatrix}\right)}

%%%%%%%%%%%%%%%%%%%%%%%%%%%%%%%%%%%%%%%%%%%%%%%%%%%%%%%

\def\Mat{\mathop\mathrm{Mat}\nolimits}
\def\Hom{\mathop\mathrm{Hom}\nolimits}
\def\rk{\mathop\mathrm{rk}\nolimits}
\def\Mod{\mbox{-}\mathrm{Mod}}
\def\Mdr{\mathrm{Mod}\mbox{-}}
\def\End{\mathop\mathrm{End}\nolimits}
\def\ker{\mathop\mathrm{Ker}\nolimits}
\def\im{\mathop\mathrm{Im}\nolimits}
\def\chr{\mathop\mathrm{char}\nolimits}
\def\Rep{\mathop\mathrm{Rep}}  
\def\pr{\mathrm{pr}}

%%%%%%%%%%%%%%%%%%%%%%%%%%%%%%%%%%%%%%%%%%%%%%%%%%%%%%

\def\iff{if and only if }
\def\wrt{with respect to }

%%%%%%%%%%%%%%%%%%%%%%%%%%%%%%%%%%%%%%%%%%%%%%%%%%%%%%

\title{Representations of Munn algebras and related semigroups}
\author[Yu. Drozd]{Yuriy A. Drozd*}
\address{National Academy of Sciences of Ukraine \\ Institute of Mathematics\\  01024 Kyiv\\ Ukraine}
\email{y.a.drozd@gmail.com (corresponding author)}
\urladdr{www.imath.kiev.ua/$\sim$drozd}
\author[A. Plakosh]{Andriana I. Plakosh}
\address{National Academy of Sciences of Ukraine \\ Institute of Mathematics\\  01024 Kyiv\\ Ukraine}
\email{andrianaplakoshmail@gmail.com}
\dedicatory{To the memory of Iosif Solomonovich Ponizovski\u{\i}}
\keywords{Munn algebras, Rees matrix semigroups, representations of valued graphs, representation types}
\subjclass{20M30, 16G60, 20M25, 16G20}

\begin{abstract}
  We establish representation types (finite, tame or wild) of finite dimensional Munn algebras with semisimple bases. As an application, we 
  establish representation types of finite Rees matrix semigroups, in particular, $0$-simple semigroups, and their mutually annihilating unions.
\end{abstract}

\maketitle

%\tableofcontents

\section*{Introduction}

 Munn algebras appeared in the theory of semigroups as semigroup algebras of \emph{completely $0$-simple} semigroups
 \cite{ClifPres1,Okninski}. They were immediately used for the study of representations of such semigroups. An important input
 was made by Ponizovski{\u\i} in the paper \cite{Poniz}, where he established the cases when a finite $0$-simple semigroup
 is \emph{representation finite}, i.e. only has finitely many indecomposable representations, over an algebraically closed 
 field $\Bbbk$ whose characteristic does not divide the order of the underlying group of its Rees matrix presentation 
 \cite[Th.\,3.5]{ClifPres1}. He also considered the case 
 of semigroups that are unions of  mutually annihilating $0$-simple semigroups with common zero. 
 
 The questions remained what happens if the field is not algebraically closed and when the representation type of such a
 semigroup is \emph{tame}, i.e. indecomposable representations of each dimension form a finite number of
 $1$-parameter families. In this article we give a complete answer to these questions (also for the fields of characteristics 
 that does non divide the orders of the underlying groups). Of course, in the case of an algebraically
 closed field our criterion of finiteness coincides with that of Ponizovski\u\i. 
 Actually, we obtain criteria of finiteness and tameness for all Munn algebras with semisimple base, 
 even in a bit more wide context than they are considered in \cite{ClifPres1}.
 To prove these results, we establish a relation of modules over Munn algebras with \emph{representations of valued graphs} 
 in the sense of \cite{DR} (in the algebraically closed case they are just \emph{representations of quivers} in the sense of \cite{GabindI}). 
 Then we apply the criteria from this paper. 
 
 It follows from \cite{tame-wild} (and can be easily checked directly) that in all other cases the Munn algebra $\mM$ (or the corresponding semigroup) is
  \emph{representation wild} over the field $\Bbbk$, i.e. for every finitely generated $\Bbbk$-algebra $A$ there is an exact functor 
  $A\Mod\to\mM\Mod$ mapping non-isomorphic modules to non-isomorphic and indecomposable to indecomposable.

 \section{Munn algebras}
 \label{s1}

 In this paper \emph{algebra} means an associative algebra over a commutative ring $\Bbbk$. We do not suppose that such an algebra is unital,
 but always suppose that modules over such algebra are also $\Bbbk$-modules and the multiplication by elements of the algebra is $\Bbbk$-bilinear.
 We denote by $A\Mod$ and $\Mdr A$, respectively, the categories of left and right $A$-modules.
 By $A^1$ we denote the algebra obtained from an algebra $A$ by the formal attachment of unit. Then the categories of $A$-modules
 and unital $A^1$-modules are equivalent. So $A$ and $B$ are Morita equivalent \iff so are $A^1$ and $B^1$. We consider the elements from $A^1$
 as formal sums $\la+ a$, where $a\in A,\,\la\in\Bbbk$.

\begin{definition}\label{def11} 
 \begin{enumerate}
 \item
  Let $R$ be a $\Bbbk$-algebra and $\mu:N\to M$ be a homomorphism of $R$-modules. Define a multiplication on $\Hom_R(M,N)$
 setting $a\cc b=a\mu b$. The resulting ring is called a \emph{Munn algebra}  and denoted by $\mM(R,M,N,\mu)$.\!%
 \footnote{\,This definition is a bit more general than that from \cite{ClifPres1} or \cite{Okninski}, where only the case of free modules is considered.}
 We say that this Munn algebra \emph{is based} on the algebra $R$.
 We denote by $\mM^1(R,M,N,\mu)$ the algebra obtained from $\mM(R,M,N,\mu)$ by the formal attachment of unit.
 \item  A Munn algebra $\mM(R,M,N,\mu)$ is said to be \emph{regular} if the homomorphism $\mu$ is von Neumann regular, i.e. there is a
 homomorphism $\theta:M\to N$ such that $\mu\theta\mu=\mu$.
For instance, this is the case if $R$ is von Neumann regular, while $M$ and $N$ are finitely generated and projective and $\mu\ne0$
(it follows from \cite[Th.\,1.7]{GdrlVNeum}).
 \end{enumerate}
\end{definition}

\begin{remark}\label{rem12} 
 One can see that $\mM(R,M,N,\mu)$ has a unit \iff there are decompositions $M\simeq M_1\+M_2$ and $N\simeq N_1\+N_2$ such that
 $\Hom_R(M_2,N)=\Hom_R(M,N_2)=0$ and the map $\bar\mu=\pr_1\cir\mu|_{N_1}$ is an isomorphism $N_1\ito M_1$. Then the unit $u:M\to N$
 coincides with $\bar{\mu}^{-1}$. Actually, in this case $\mM(R,M,N,\mu)\simeq\mM(R,M_1,N_1,\bar{\mu})\simeq\End_RM_1$.
% (We do not use this fact.)
\end{remark}

\begin{proposition}\label{prop13} 
 Let $\mM(R,M,N,\mu)$ be a regular Munn algebra. There are isomorphisms $M\simeq L\+M'$ and $N\simeq L\+N'$ such that \wrt these decompositions
 $\mu=\smtr{1_L&0\\0&0}$.
\end{proposition}

\begin{proof}
Let $\theta:M\to N$ be such that $\mu\theta\mu=\mu$. Then $\mu\theta:M\to M$ and $\theta\mu:N\to N$ are idempotents. Therefore,
$M=M_1\+M_2$, where $M_1=\im\mu\theta,\,M_2=\ker\mu\theta$ and $N=N_1\+N_2$, where $N_1=\im\theta\mu,\,N_2=\ker\theta\mu$.
One easily sees that $\ker\mu=\ker\theta\mu$ and $\im\mu=\im\mu\theta$, so $\bar{\mu}=\pr_1\cir\mu|_{N_1}$ is an isomorphism and 
$\bar{\mu}^{-1}=\pr_1\cir\theta|_{M_1}$, while $\mu|_{N_2}=0$ and $\pr_2\cir\mu=0$, hence
$\mu=\smtr{\bar\mu &0\\0&0}$ \wrt these decompositions. Obviously, it implies the claim.
\end{proof}

\begin{definition}\label{def14} 
 We write $\mM(R,L,M,N)$ instead of $\mM(R,L\+M,L\+N,\mu)$, where $\mu=\smtr{1_L&0\\0&0}$, and call such a Munn algebra \emph{normal}. 
 Thus every regular Munn algebra is isomorphic to a normal one. As above, we denote by $\mM^1(R,L,M,N)$ the algebra obtained 
 from $\mM(R,L,M,N)$ by the formal attachment of unit.
\end{definition}

\begin{lemma}\label{lem15} 
 Let $A$ and $C$ be two rings, $P$ be a right $C$-module, $M$ be a right $A$-module and $N$ be a right $A$--left $C$-bimodule.
 Define the natural map $\phi:P\*_C\Hom_A(M,N)\to\Hom_A(M,P\*_CN)$ mapping $p\*f$ to the homomorphism $x\mapsto p\*f(x)$.
 If $P$ is projective and either $P$ or $M$ is finitely generated, $\phi$ is an isomorphism.
\end{lemma}
 The proof is obvious. \qed

\begin{lemma}\label{lem16} 
 Let $A$ be a unital ring, $1=e_1+e_2$, where $e_1,e_2$ are orthogonal idempotents. We denote $A_i=e_iA$, $A_{ij}=e_iAe_j\simeq\Hom_A(A_j,A_i)$
 and identify $A$ with the ring of matrices
 \begin{equation}\label{eq12} 
   \mtr{A_{11} & A_{12} \\ A_{21} & A_{22} }.
 \end{equation}
 Let $P$ be a progenerator of the category $\Mdr A_{11}$. 
 Then $P^\sharp=(P\*_{A_{11}}A_1)\+A_2$ is a progenerator of the category $\Mdr A$, hence $A\Mod\simeq B\Mod$, 
 where $B=\End_AP^\sharp$. The ring $B$ can be identified with the ring of matrices
 \begin{equation}\label{eq11} 
 B=  \mtr{ B_{11}  & B_{12} \\ B_{21} & B_{22} },
 \end{equation}
 where
 $
  B_{11}=\End_{A_{11}}P,\,
  B_{12}=P\*_{A_{11}}A_{12},\,
  B_{21}=\Hom_{A_{11}}(P,A_{21}),\,
  B_{22}=A_{22}.
 $
\end{lemma} 
\begin{proof}
For some $m$ there is an epimorphism of $A_{11}$-modules $P^m\tto A_{11}$, which induces an epimorphism $(P\*_{A_{11}}A_1)^m\tto A_1$.
Hence, $A$ is a direct summand of $(P\*_{A_{11}}A_1)^m\+A_2$ and $P^\sharp$ is a progenerator of $A\Mod$. 
Using Lemma~\ref{lem15}, we obtain:
\begin{align*}
 & \Hom_A(P\*_{A_{11}} A_1,P\*_{A_{11}} A_1)\simeq \\
 & \simeq \Hom_{A_{11}}(P,\Hom_A(A_1,P\*_{A_{11}}A_1) \simeq \\
 & \simeq\Hom_{A_{11}}(P,P\*_{A_{11}}A_{11})\simeq  \End_{A_{11}}P;\\
 & \Hom_A(A_2,P\*_{A_{11}}A_1)\simeq P\*_{A_{11}}A_{12};\\
 & \Hom_A(P\*_{A_{11}}A_1,A_2)\simeq \Hom_{A_{11}}(P,A_{21}).
\end{align*}
 It gives the presentation \eqref{eq11} for $\End_AP^\sharp$.
\end{proof}

\begin{theorem}\label{th17} 
   Let $\mM=\mM(R,L,M,N)$ be a normal Munn algebra, $C=\End_RL$ and $P$ be a progenerator of the category $\Mdr C$.
   Then $\mM$ is Morita equivalent to the normal Munn algebra $\mM(R,P\*_CL,M,N)$.
\end{theorem}
\begin{proof}
 Let $A=\mM^1(R,L,M,N)$. Consider the idempotents $e_1=\smtr{1&0\\0&0}$ and $e_2=1-e_1$. The presentation
 \eqref{eq12} of the algebra $A$ is of the form
 \begin{equation}\label{eq13} 
  \mtr{ C & \Hom_R(M,L) \\ \Hom_R(L,N) & \Bbbk+\Hom_R(M,N) }
 \end{equation}
  By Lemma~\ref{lem16}, $A$ is Morita equivalent to the algebra $B$ of the matrices of the form \eqref{eq11}, where,
  due to Lemma~\ref{lem15},
  \begin{align*}
   B_{11}&=\Hom_C(P,P)\simeq \Hom_C\big(P,P\*_C\Hom_R(L,L)\big)\simeq \\
   			&\simeq \Hom_C(P,\Hom_R(L,P\*_CL))\simeq \Hom_R(P\*_CL,P\*_CL);\\
   B_{12}&=P\*_C\Hom_R(M,L)\simeq \Hom_R(M,P\*_CL);\\
   B_{21}&=\Hom_C(P,\Hom_R(L,N))\simeq\Hom_R(P\*_CL,N);\\
   B_{22}&=\Bbbk+ \Hom_R(M,N).  			
  \end{align*}
  But it is just the matrix presentation of $\mM^1(R,P\*_CL,M,N)$.
  \end{proof}
 
  The following fact is evident.
  
  \begin{proposition}\label{prop18} 
   $\prod_{k=1}^s\mM(R_k,M_k,N_k,\mu_k)\simeq \mM(R,M,N,\mu)$, where 
   $R=\prod_{k=1}^sR_k,\,M=\bigoplus_{k=1}^sM_k,\,N=\bigoplus_{k=1}^sN_k$ and $\mu|_{N_k}=\mu_k$.
  \end{proposition}
  
  \begin{remark}\label{rem19} 
   Note that $\prod_{k=1}^s\mM^1(R_k,M_k,N_k,\mu_k)\not\simeq \mM^1(R,M,N,\mu)$.
  \end{remark}
 
 Let now $R$ be a semisimple ring. Then $R=\prod_{k=1}^sR_k$, where $R_k=\Mat(d_k,F_k)$ for some
 integers $d_k$ and some skewfields $F_k$. So any Munn algebra based on $R$ is a product of Munn algebras based on the simple
 algebras $R_k$. All of them are regular, so can be supposed normal. 
 
 \begin{proposition}\label{prop110} 
    Let $R=\Mat(d,F)$, where $F$ is a skewfield, $U$ be the simple $R$-module, $L=U^r,\,M=U^m,\,N=U^n$. The algebra $\mM(R,L,M,N)$,
  up to isomorphism, only depends on $r,m,n$ and does not depend on $d$. In particular, it is isomorphic to $\mM(F,F^r,F^m,F^l)$.
  
  \smallskip
  \emph{We denote the algebra $\mM(F,F^r,F^m,F^n)$ by $\mM(F,r,m,n)$.\!}%
   \footnote{\,Ponizovski{\u\i} \cite{Poniz} denotes this algebra by $\mathfrak{A}(E_{m+r,n+r,r},F)$.}
  \end{proposition} 
  \begin{proof}
  Indeed, $\Hom_R(U^k,U^l)\simeq \Mat(l\xx k,F)$ does not depend on $d$ and \wrt such isomorphisms $\mM(R,L,M,N)=\Mat((r+n)\xx(r+m),F)$
  with the multiplication $a\cc b=a\mu b$, where $\mu=\smtr{I&0\\0&0}$ (of size $(r+m)\xx(r+n)$) and $I$ is the $r\xx r$ unit matrix.
  \end{proof} 
 
 \begin{theorem}\label{th111} 
  Let $\mM=\prod_{k=1}^s\mM(F_k,r_k,m_k,n_k)$, where $F_k$ are skewfields. Then $\mM$ is Morita equivalent to
  $\prod_{k=1}^s\mM(F_k,1,m_k,n_k)$. 
  \end{theorem}
 \begin{proof}
  Let $R=\prod_{k=1}^sF_k$, $L_k=F_k^{r_k}$ and $L=\prod_{k=1}^sL_k$. Then $C_k=\End_RL_k\simeq\Mat(r_k\xx r_k,F_k)$. Let $P_k$ be the 
  simple right $C_k$-module. It is a progenerator of the category $\Mdr C_k$ and $P_k\*_{C_k}L_k\simeq F_k$. Now apply Theorem~\ref{th17}.
 \end{proof}
  
  We denote the algebra $M(F,1,m,n)$ by $\mM(F,m,n)$. It is the algebra of $(n+1)\xx(m+1)$ matrices over $F$ with the multiplication
  $a\cc b=a\mu b$, where $\mu$ is the $(m+1)\xx(n+1)$ matrix with $1$ at the $(1,1)$-place and $0$ elsewhere.

 \section{Representations}
 \label{s2}

   In this section we consider  representations of finite dimensional regular Munn algebras over a field $\Bbbk$ with a semisimple base. According to 
   Theorem~\ref{th111}, such an algebra is Morita equivalent to a direct product $\mM=\prod_{k=1}^s\mM_k$, where 
   $\mM_k=\mM(F_k,m_k,n_k)$ and $F_k$ are skewfields. 
   If $m_k=n_k=0$, $\mM(F_k,m_k,n_k)=F_k$ and is a direct factor of $\mM^1$. So we can and will 
   suppose that there are no such components in $\mM$. The algebra $\mM_k$ contains an idempotent $e_k$ which is the
   $(n_k+1)\xx(m_k+1)$ matrix with $1$ at the $(1,1)$-place and $0$ elsewhere. Let $e_0=1-\sum_{k=1}^se_k$. Then,
   if $k\ne0$, $e_k\mM^1e_k=F_k$, $e_0\mM^1e_k=F_k^{n_k}$, $e_k\mM^1e_0=F_k^{m_k}$, 
   $e_0\mM^1e_0=\Bbbk+\bigoplus_{k=1}^sM_k$, where $M_k\simeq\Mat(n_k\xx m_k,F_k)$, and $e_k\mM^1e_l=0$ if $0\ne k\ne l\ne0$. 
   Choose an $F_k$-basis $\{a_{k1},a_{k2},\dots,a_{km_k}\}$ in each space $e_k\mM^1e_0$ and an $F_k$-basis $\{b_{k1},b_{k2},\dots,b_{kn_k}\}$
   in each space $e_0\mM^1e_k$. Then $a_{ki}b_{lj}=0$ for all $k,l,i,j$, $b_{ki}a_{lj}=0$ if $k\ne l$ and $\{b_{ki}a_{kj}\}$ is a basis of $M_k$.
   For every $\mM^1$-module $V$ set $V_k=e_kV\ (0\le k\le s)$. It is a vector space over $F_k$. The multiplication
   by $a_{ki}$ gives rise to a $\Bbbk$-linear map $\al_{ki}:V_0\to V_k$ and the multiplication by $b_{kj}$ gives rise to a $\Bbbk$-linear map
   $\be_{ki}:V_k\to V_0$. Since $\Hom_\Bbbk(V_0,V_k)\simeq\Hom_{F_k}(F_k\*_\Bbbk V_0,V_k)$ and
    $\Hom_\Bbbk(V_k,V_0)\simeq\Hom_{F_k}(V_k,\Hom_\Bbbk(F_k,V_0))$, both $\al$ and $\be$ can be considered as matrices over $F_k$
    of appropriate sizes. So $V$ is defined by the set of maps (or of matrices) $\{\al_{ki},\be_{lj}\}$ such that $\al_{ki}\be_{lj}=0$ for all $k,l,i,j$. 
    We present it by the diagram
    \[
      V:\   \xymatrix{  \{ V_k\}\,{\stackrel0\circlearrowleft}  \ar@<-.5ex>[rr]_{\{\be_{kj}\}} && V_0, \ar@<-.5ex>[ll]_{\{\al_{ki}\}}  }
    \]
    
   A homomorphism $\phi:V\to V'$ is given by a set of $F_k$-linear maps $\phi_k:V_k\to V'_k\ (0\le k\le s)$, where $F_0=\Bbbk$,
    such that $\phi_k\al_{ki}=\al'_{ki}\phi_0$ and $\phi_0\be_{kj}=\be'_{kj}\phi_k$, i.e. the following diagram is commutative:
    \begin{equation}\label{eq22} 
      \vcenter{  \xymatrix{  \{ V_k\}\,{\stackrel0\circlearrowleft} \ar@<-.5ex>[rr]_{\{\be_{kj}\}} \ar[d]_{\{\phi_k\}} && V_0 \ar@<-.5ex>[ll]_{\{\al_{ki}\}}  \ar[d]^{\phi_0}	\\ 
   									\{ V'_k\}\,{\stackrel0\circlearrowleft}  \ar@<-.5ex>[rr]_{\{\be'_{kj}\}}  && V'_0 \ar@<-.5ex>[ll]_{\{\al'_{ki}\}}  	 }    }
    \end{equation}
   $\phi$ is an isomorphism \iff so are all $\phi_k$. 
   
   Set $V_+=\sum_{l,j}\im\be_{lj}\sbe V_0,\ V_-=V_0/V_+$. Then $\al_{ki}(V_+)=0$. Hence $\al_{ki}$ can be considered as a map $V_-\to V_k$
   and we obtain a diagram
   \[
   \tiv:\ \vcenter{    \xymatrix@R=1ex{ && V_- \ar[dll]_{\{\al_{ki}\}} \\ \{V_k\} \ar[drr]_{\{\be_{kj}\}} \\ && V_+	  } }
   \]
   with the condition $\sum_{k,j}\im\be_{kj}=V_+$.  Such diagram can be considered as a representation of the \emph{realization 
   $(\dM,\Om)$ of the valued graph} $(\Ga,d)$ in the sense of \cite{DR}. Namely the vertices of the graph $\Ga$ are 
   $\{+,-,1,2,\dots,s\}$, $d_k=\dim_\Bbbk F_k$, $d_{k+}=(m_k,m_kd_k)$, $d_{-k}=(n_kd_k,n_k)$ and $d_{ij}=0$ otherwise.
   The \emph{orientation} $\Om$ of the edge $\{k,+\}$ is $\,k\to+\,$ and that of the edge $\{-,k\}$ is $\,-\to k$. The \emph{modulation} $\dM$
   of $\Ga$ is given by the algebras $F_k$ and $F_\pm=\Bbbk$, $F_k\mbox{-}F_-$-bimodules ${_kM_-}=m_kF_k$ and
   $F_+\mbox{-}F_k$-bimodules ${_+M_k}=n_kF_k$. Thus a representation of this realization is indeed given by a set of 
   $F_k$-vector spaces $V_k$, $F_0$-vector spaces $V_{\pm}$ and a set of linear maps $\tal_k:n_kV_-\to V_k$ and $\tbe_l:m_lV_l\to V_+$.
   There components are just $\al_{ki}$ and $\be_{lj}$. 
   
  \begin{theorem}\label{th21} 
  Let $\Rep^+(\dM,\Om)$ be the full subcategory of the category of representations of $(\dM,\Om)$ such that 
  $\sum_{l=1}^s\im\tbe_l=V_+$ and $\bigcap_{k}\ker\tal_k=0$.
     Let also $\mM\Mod^+$ be the full subcategory of $\mM\Mod$ consisting of such modules $V$ that 
     $\sum_{l,j}\im\be_{lj}=\bigcap_{k,i}\ker\al_{ki}$. Denote by $\kI$ the ideal of the category $\mM\Mod^+$ consisting of all morphisms
     $\phi:V\to V'$ such that $\phi_k=0$ for $k\ne0$, $\phi_0(V_+)=0$ and $\im\phi_0\sbe V'_+$.    
     Then $\mM\Mod^+/\kI\simeq\Rep^+(\dM,\Om)$ and $\kI^2=0$.
 \end{theorem}
 \begin{proof}
    We have already constructed, for any $\mM$-module $V$, the representation $\tilde{V}$. By definition, $\tilde{V}\in\Rep^+(\dM,\Om)$.
   Given a homomorphism $\phi=\{\phi_k\}:V\to V'$ as in \eqref{eq22}, we obtain linear maps $\phi_+:V_+\to V'_+$ and $\phi_-:V_-\to V'_-$
   such that together with the maps $\phi_k$ they give a morphism $\tph:\tiv\to\tiv'$. Obviously, $\tph=0$ \iff $\phi\in\kI$.
   Thus we obtain a functor $\Phi:\mM\Mod^+/\kI\to\Rep^+(\dM,\Om)$.  Obviously $\kI^2=0$.
   
   Let $W=(W_k,W_+,W_-,\al_k,\be_k \mid 1\le k\le s)$ be a representation from $\Rep^+(\dM,\Om)$.  Set $\tiw_0=W_+\+W_-$, 
    take for $\tal_{ki}:W_0\to W_k$ the maps that are $0$ on $W_+$ and coincide with the components of $\al_k$ on $W_-$, and take 
    for $\tbe_{lj}:W_l\to \tiw_0$ the components of $\be_l:W_l\to W_+$. It defines an $\mM$-module $\tiw\in\mM\Mod^+$. If
    $\psi:W\to W'$ is a morphism of representations, set $\tps(w)=\psi_+(w_+)+\psi_-(w_-)$ if $w=w_++w_-$, where $w_\pm\in W_\pm$.
    It gives a homomorphism $\tps:\tiw\to \tiw'$. Taking its class modulo $\kI$, we obtain a functor $\Psi:\Rep^+(\dM,\Om)\to\mM\Mod^+/\kI$.
    One easily verifies that this functor is quasi-inverse to $\Phi$.
      \end{proof}
       
   \begin{remark}\label{rem22} 
     Since $\kI^2=0$, the isomorphism classes of objects in $\mM\Mod^+$ are the same as in $\mM\Mod^+/\kI$.
         The only indecomposable representations not belonging to $\Rep^+(\dM,\Om)$ are two \emph{trivial representations} such that $V_+=\Bbbk$
         (or $V_-=\Bbbk$) and $V_k=0$ for $k\ne+$ (respectively, for $k\ne-$). 
         The only indecomposable $\mM$-module not belonging to $\mM\Mod^+$ is the $1$-dimensional vector space
   with zero multiplication by the elements of $\mM$. Therefore, the \emph{representation type} of the algebra $\mM$ (finite, tame or wild) 
   is the same as that of the realization $(\dM,\Om)$ of the valued graph $\Ga$.
      \end{remark}  
       
   It is proved in \cite{DR} that the representation type of $(\dM,\Om)$ actually only depends on the valued graph itself. 
   Namely, it is representation finite \iff all its connected components are \emph{Dynkin graphs} and representation tame
    \iff all of them are Dynkin or \emph{Euclidean} (\emph{extended Dynkin}) graphs and at least one Euclidean graph occurs. 
    For the list of these graphs see \cite[p.\,3]{DR}. In all other cases it is representation wild.
   
   Taking into account the construction of the valued graph $\Ga$ from the algebra $\mM$, we can establish the representation type of any
   finite dimensional Munn algebra with a semisimple base. Actually it only depends on the set of triples $\{(d_k,m_k,n_k)\}$, where $d_k=\dim_\Bbbk F_k$. 
   We use the following notations:   
   \begin{align*}
    & \dT(d_1,\dots,d_r\mid d_{r+1},\dots,d_s)=\\
    &\ =\{(d_1,1,0),\dots,(d_r,1,0),(d_{r+1},0,1),\dots,(d_s,0,1)\},\\
    \intertext{and, for $\dT=\dT(d_1,\dots,d_r\mid d_{r+1},\dots,d_s)$,}
    & S^-(\dT)={\sum}_{k=1}^rd_k,\\
    & S^+(\dT)={\sum}_{k=r+1}^sd_k,\\
    & S(\dT)=S^-(\dT)+S^+(\dT).
   \end{align*}
   Certainly, maybe $r=0$ or $r=s$.
 
 \begin{theorem}\label{th23} 
 Let $\mM=\prod_{k=1}^s\mM(F_k,m_k,n_k)$, $\dT=\{(d_k,m_k,n_k)\mid (m_k,n_k)\ne(0,0)\}$, where $d_k=\dim_\Bbbk F_k$.
 \begin{enumerate}
 \item\hspace*{-1ex}%
  \footnote{\,If the field $\Bbbk$ is algebraically closed, hence all $d_k=1$, this result coincides with that of Ponizovski{\u\i} \cite[n$^\circ$\,5]{Poniz}.}
   $\mM$ is representation finite \iff $\dT=\dT_0\cup\dT_1$, where $\dT_0=\dT(d_1,\dots,d_r\mid d_{r+1},\dots,d_s)$ for
 some $d_k$ and
	\begin{enumerate}
 		\item  either $\dT_1=\0$ and  $\max\{S^-(\dT_0),S^+(\dT_0)\}\le3$
 		\item  or $\dT_1=\{(1,1,1)\}$, $S(\dT_0)\le3$ and $\max\{S^-(\dT_0),S^+(\dT_0)\}\le2$.
 		\end{enumerate}
 	\smallskip
 \item  $\mM$ is representation tame \iff $\dT=\dT_0\cup\dT_1$, where 
 $\dT_0=\dT(d_1,\dots,d_r\mid d_{r+1},\dots,d_s)$ for some $r$ and $d_k$, and
		\begin{enumerate}
		 \item either $\dT_0=\0$ and $\dT$ is one of the sets $$\{(1,1,1),(1,1,1)\},\,\{(2,1,1)\},\,\{(1,2,0)\},\,\{(1,0,2)\},$$
		 \item  or $\dT_1=\0$ and $\max\{S^-(\dT_0),S^+(\dT_0)\}=4$,
		 \item  or $\dT_1=\{(1,1,1)\}$ and $S^-(\dT_0)=S^+(\dT_0)=2$.
		 \end{enumerate}
  \smallskip
  \item  In all other cases $\mM$ is representation wild.
 \end{enumerate}	  
 
  \end{theorem}
  \begin{proof}
  (1a) In this case the graph $\Ga$ is a disjoint union of  $2$ graphs of the types $A_2,A_3,D_4,B_2$ or $B_3$.
  
  (1b) In this case $\Ga$ is of one of the types $A_3,A_4,A_5,D_5,D_6,B_4$ or $B_5$.
  
  In other cases $\Ga$ is not a disjoint union of Dynkin graphs.
  
  From now on we only list the cases when $\mM$ is not representation finite.
  
  (2a) In these cases $\Ga$ is, respectively, of type $\tilde{A}_3$, or $\tilde{B}_2$, or $\tilde{A}_{12}$.
  
  (2b) In this case $\Ga$ is a disjoint union of two graphs, where either both are of types $\tilde{D}_4,\tilde{BD}_3,\tilde{B}_2,\tilde{A}_{11}$ or $\tilde{G}_2$
  or one is of one of these types while the other is of a type cited in case (1a). 
  
  (2c) In this case $\Ga$ is of type $\tilde{D}_6,\tilde{BD}_5$ or $\tilde{B}_4$.
   
   (3) In all other cases the graph $\Ga$ is not a disjoint union of Dynkin and Euclidean graphs.
  \end{proof}

 \section{Semigroups}
 \label{s3}

 We apply the obtained result to representations of finite \emph{Rees matrix semigroups}. Recall 
  \cite[\S3.1]{ClifPres1} that such semigroup $\kM(G,p,q,\mu)$ is given by a finite group $G$ and a matrix $\mu$ of size $p\xx q$
 with coefficients from the group $G$.
  The elements of $\kM(G,p,q,\mu)$ are $q\xx p$ matrices with coefficients from $G^0=G\sqcup\{0\}$ containing at most one non-zero 
  element and the multiplication is defined by the rule $a\cc b=a\mu b$. If the sandwich matrix $\mu$ is \emph{regular}, i.e. 
  every column and every row of $\mu$ contains a non-zero element, the semigroup $\kM(G,p,q,\mu)$ is \emph{$0$-simple} 
  (hence \emph{completely $0$-simple}) and every finite $0$-simple semigroup is isomorphic to a Rees matrix semigroup with a 
  regular sandwich matrix \cite[Th.3.5]{ClifPres1}. We always suppose that the matrix $\mu$ is non-zero; otherwise $\kM(G,p,q,\mu)$
  is just a semigroup with zero multiplication.
  
  Let $\Bbbk$ be a field, $R=\Bbbk G$ and $\kM=\kM(G,p,q,\mu)$. Obviously, $\Bbbk\kM=\mM(R,R^p,R^q,\mu)$, where $\mu$ is 
  considered as an element of $\Mat(p\xx q,R)$ and is identified with an $R$-homomorphism $R^q\to R^p$.
  We suppose that $\chr\Bbbk\nmid \#(G)$. Then $R$ is semisimple. Namely,
  let $U_1,U_2,\dots,U_s$ be all irreducible representations of $G$ over $\Bbbk$, $F_k=\End_GU_k$, $d_k=\dim_\Bbbk F_k$ and 
  $u_k=\dim_\Bbbk U_k$. Set $c_k=\frac{u_k}{d_k}$. Then $R\simeq\prod_{k=1}^sR_k$, where $R_k=\Mat(c_k\xx c_k,F_k)$,  and
  $\Mat(p\xx q,R_k)= \Mat(pc_k\xx qc_k,F_k)$.
  Denote by $\mu_k$ the projection of $\mu$ onto $\Mat(pc_k\xx qc_k,F_k)$ and set $r_k=\rk\mu_k$. 
   As $\mu\ne0$, also all $\mu_k\ne0$ and the Munn algebra $\Bbbk\kM$ is regular.
  Then $\Bbbk\kM\simeq\prod_{k=1}^s\mM(F_k,r_k,m_k,n_k)$, where $m_k=pc_k-r_k$ and $n_k=qc_k-r_k$. 
 
  Theorem~\ref{th111} now implies the following result.
  
  \begin{corollary}\label{cor31} 
  $\Bbbk\kM$ is Morita equivalent to $\prod_{k=1}^s\mM(F_k,m_k,n_k)$.
    \end{corollary}  
    
  \begin{remark}\label{rem32} 
   Note that $c_k\mid m_k-n_k$ and $\frac{m_k-n_k}{c_k}=p-q$ does not depend on $k$. In particular, if $m_k=n_k$, or $m_k>n_k$, or $m_k<n_k$
   for some $k$, the same holds for all $k$.
  \end{remark}
 
   From Corollary~\ref{cor31} and Theorem~\ref{th23}, taking into account Remark~\ref{rem32}, we obtain a classification of 
   representation types of Rees matrix semigroups, in particular, of $0$-simple semigroups. In the next theorem we use the 
   just introduced notations.
   
   \begin{theorem}\label{th33} 
    Let $\kM=\kM(G,p,q,\mu)$ be a finite Rees matrix semigroup, $\Bbbk$ be a field such that 
    $\chr\Bbbk\nmid \#(G)$. Set $\dT(\kM)=\{(d_k,m_k,n_k)\mid (m_k,n_k)\ne(0,0)\}$.
    \begin{enumerate}
    \item\hspace*{-1ex}%
  \footnote{\,If the field $\Bbbk$ is algebraically closed, hence all $d_k=1$, this result was proved by Ponizovski{\u\i} \cite{Poniz}.}
    $\kM$ is representation finite over the field $\Bbbk$ \iff
    		\begin{enumerate}
    		\item either $\dT=\{(1,1,1)\}$
    		\item or $\,\#(G)\le 3$ and $\dT$ contains either only triples $(d_k,1,0)$ or only triples $(d_k,0,1)$.
    		\end{enumerate}
    \item $\kM$ is representation tame over the field $\Bbbk$ \iff 
    		\begin{enumerate}
    			\item either $\dT(\kM)=\{(1,1,1),(1,1,1)\}$, or $\dT(\kM)=\{(2,1,1)\}$,
    			\item  or $\,\#(G)=4$ and  $\dT(\kM)$ contains either only triples $(d_k,1,0)$ or only triples $(d_k,0,1)$,
    			\item  $G=\{1\}$ and $\dT(\kM)=\{(1,2,0)\}$ or $\dT(\kM)=\{(1,0,2)\}$.
    		\end{enumerate}
    		\smallskip
    \item In all other cases $\kM$ is representation wild over the field $\Bbbk$.
    \end{enumerate} 
    	Note that in cases (1a) and (2a) $p=q$, while in cases (1b) and (2b) the group $G$ is commutative.
   \end{theorem}
   
   \begin{remark}\label{r3} 
    According to Proposition~\ref{prop110}, the algebra $\Bbbk\kM(G,p,q,\mu)$ only depends on the ranks $r_k$. Elementary transformations
    of the matrix $\mu$ do not change these ranks. Obviously, using them one can obtain a matrix $\mu'$ such that there is a non-zero element
    in every row and in every column. Therefore, $\Bbbk\kM(G,p,q,\mu)\simeq\Bbbk\kM(G,p,q,\mu')$ and $\kM(G,p,q,\mu')$ is a 0-simple
    semigroup \cite[Thm.3.3]{ClifPres1}. Thus, for every Rees matrix semigroup with a non-zero sandwich matrix there is a 0-simple semigroup 
    with the same representation theory.
   \end{remark}
    
   If a finite semigroup $\kS=\bigvee_{i=1}^t\kM_i$ is a union of pairwise annihilating Rees matrix semigroups $\kM_i$ with common $0$, 
   its semigroup algebra $\Bbbk\kS$ is a direct product of semigroup algebras $\Bbbk\kM_i$ and all of them are Munn algebras. 
   So we obtain the following result.
   
   \begin{theorem}\label{th34} 
   Let $\kS=\bigvee_{i=1}^t\kM_i$, where $\kM_i=\kM(G_i,m_i,n_i,\mu_i)$ are finite Rees matrix semigroups, $\Bbbk$ be a field such that 
    $\chr\Bbbk\nmid \#(G_i)$ for all $i$. Denote
   \begin{align*}
    & T_>=\sum_{m_i>n_i} \#(G_i),\\
    & T_<=\sum_{m_i<n_i} \#(G_i),\\
    &  \dT_0=\bigcup_{m_i\ne n_i}\dT(\kM_i),\\
    & \dT_1=\bigcup_{m_i= n_i}\dT(\kM_i)
   \end{align*}
   \begin{enumerate}
   \item\hspace*{-1ex}%
  \footnote{\,If the field $\Bbbk$ is algebraically closed, this result easily follows from that of Ponizovski{\u\i} \cite[\rm{n}$^\circ$5]{Poniz} and Remark
  \ref{rem32}.}
      $\kS$ is representation finite over the field $\Bbbk$ \iff
 		\begin{enumerate}
 		\item either  $\dT_1=\0$, $\max\{T_>,T_<\}\le3$ and all triples from $\dT_0$ are either $(d_k,1,0)$ or $(d_k,0,1)$ 
 		\item or $\dT_1=\{(1,1,1)\}$, $T_>+T_<\le 3$, $\max\{T_>,T_<\}\le2$ and all triples from $\dT_0$ are either $(d_k,1,0)$ or $(d_k,0,1)$.
 		\end{enumerate}
 	\item   $\kS$ is representation tame over the field $\Bbbk$ \iff
 		\begin{enumerate}
 		\item either $\dT_1=\0$, $\max\{T_>,T_<\}= 4$ and all triples from $\dT_0$ are either $(d_k,1,0)$ or $(d_k,0,1)$,
 		\item or $\dT_1=\{(1,1,1)\}$, $T_>=T_<= 2$ and all triples from $\dT_0$ are either $(d_k,1,0)$ or $(d_k,0,1)$,
 		\item or $\dT_0=\0$ and either $\dT_1=\{(1,1,1),(1,1,1)\}$ or $\dT_1=\{(2,1,1)\}$,
 		\item or $\dT_1=\0$ and $\dT_0=\{(1,2,0)\}$ or $\dT_0=\{(1,0,2)\}$.
 		\end{enumerate}
 		In the last case there is a unique index $i$ such that $m_i\ne n_i$ and the corresponding group $G_i=\{1\}$.
 	\item  In all other cases $\kS$ is representation wild over the field $\Bbbk$.
   \end{enumerate}   
   \end{theorem}
   
   \section*{Acknowledgements}
   \footnotesize
   
   This work was supported within the framework of the program of support of priority for the state scientific researches and scientific and technical (experimental) developments of the Department of Mathematics NAS of Ukraine for 2022-2023 (Project ``Innovative methods in the theory of differential equations, computational mathematics and mathematical modeling'', No. 7/1/241). The final version of the paper was prepared during the stay of the first author in the Max-Plank-Institute for
Mathematics (Bonn) and he is grateful to the Institute for their kind support. %The authors are also grateful to the referee for his careful reading the text.
  
\bibliographystyle{acm}
\bibliography{DrozdPlakoshMunn}

\end{document}